\documentclass[12pt]{article}

\usepackage{latexsym,amsmath,amscd,amssymb,graphics}
\usepackage{enumerate}

\usepackage{graphicx}
\usepackage{eucal}
\usepackage[colorlinks]{hyperref}
\usepackage{url}
\usepackage{framed}

\usepackage[all]{xy}

\makeatletter

\@addtoreset{figure}{section}
\def\thefigure{\thesection.\@arabic\c@figure}
\def\fps@figure{h, t}
\@addtoreset{table}{bsection}
\def\thetable{\thesection.\@arabic\c@table}
\def\fps@table{h, t}
\@addtoreset{equation}{section}

\makeatother

\newenvironment{proof}[1][Proof]{\noindent\textbf{#1.} }{\ \rule{0.5em}{0.5em}}

\begin{document}

\newtheorem{theorem}{Theorem}[section]
\newtheorem{definition}[theorem]{Definition}
\newtheorem{lemma}[theorem]{Lemma}
\newtheorem{remark}[theorem]{Remark}
\newtheorem{proposition}[theorem]{Proposition}
\newtheorem{corollary}[theorem]{Corollary}
\newtheorem{example}[theorem]{Example}

\newcommand \al{\alpha}
\newcommand\be{\beta}
\newcommand\ga{\gamma}
\newcommand\de{\delta}
\newcommand\ep{\varepsilon}
\newcommand\ze{\zeta}
\newcommand\et{\eta}
\renewcommand\th{\theta}
\newcommand\io{\iota}
\newcommand\ka{\kappa}
\newcommand\la{\lambda}
\newcommand\rh{\rho}
\newcommand\ro{\rho}
\newcommand\si{\sigma}
\newcommand\ta{\tau}
\newcommand\ph{\varphi}
\newcommand\ch{\chi}
\newcommand\ps{\psi}
\newcommand\om{\omega}
\newcommand\Ga{\Gamma}
\newcommand\De{\Delta}
\newcommand\Th{\Theta}
\newcommand\La{\Lambda}
\newcommand\Si{\Sigma}
\newcommand\Ph{\Phi}
\newcommand\Ps{\Psi}
\newcommand\Om{\Omega}

\renewcommand\L{\on{pounds}}
\newcommand\resp{resp.\ }
\newcommand\ie{i.e.\ }
\newcommand\cf{cf.\ }
\newcommand\oo{{\infty}}
\newcommand\totoc{}
\renewcommand\o{\circ}
\renewcommand\div{\on{div}}
\newcommand\x{\times}
\newcommand\on{\operatorname}
\newcommand\Ad{\on{Ad}}
\newcommand\ad{\on{ad}}
\newcommand\grad{\on{grad}}\newcommand\ver{\on{Ver}}
\newcommand\Emb{\on{Emb}}
\newcommand\ev{\on{ev}}
\newcommand\Den{\on{Den}}
\newcommand\pr{\on{pr}}
\newcommand\Aut{\mathcal{A}ut}
\newcommand\F{\mathcal{F}}
\renewcommand\H{\mathcal{H}}
\newcommand\Diff{\on{Diff}}
\newcommand\Gr{\on{Gr}}
\newcommand\hor{\on{hor}}
\newcommand\End{\on{End}}
\renewcommand\Im{\on{Im}}
\renewcommand\L{\on{pounds}}
\newcommand\Ham{\on{Ham}}\newcommand\Hom{\on{Hom}}
\newcommand\g{\mathfrak g}
\newcommand\pa{\partial}
\newcommand\h{\mathfrak h}
\newcommand\ou{\mathfrak o}
\newcommand\dd{{\mathbf d}}
\newcommand\ii{{\mathbf i}}
\newcommand\QQ{\mathbf{Q}}
\newcommand\PP{\mathbf{P}}
\renewcommand\v{\mathbf{v}}
\renewcommand\u{\mathbf{u}}
\newcommand\J{\mathbf{J}}
\newcommand\X{\mathfrak X}
\renewcommand\P{\mathcal P}
\renewcommand\O{\mathcal O}
\newcommand\A{\mathcal A}
\newcommand\ZZ{\mathbb Z}
\newcommand\RR{\mathbb R}

\newcommand{\fint}{-\!\!\!\!\!\!\int}
\def\Xint#1{\mathchoice
{\XXint\displaystyle\textstyle{#1}}%
{\XXint\textstyle\scriptstyle{#1}}%
{\XXint\scriptstyle\scriptscriptstyle{#1}}%
{\XXint\scriptscriptstyle\scriptscriptstyle{#1}}%
\!\int}
\def\XXint#1#2#3{{\setbox0=\hbox{$#1{#2#3}{\int}$ }
\vcenter{\hbox{$#2#3$ }}\kern-.5\wd0}}
\def\ddashint{\Xint=}
\def\dashint{\Xint-}

\def\below#1#2{\mathrel{\mathop{#1}\limits_{#2}}}


\title{The group structure\\ for jet bundles over Lie groups}

\author{Cornelia Vizman}

\date{  }

\maketitle

\makeatother

\maketitle

\begin{abstract}
The jet bundle $J^kG$ of $k$-jets of curves in a Lie group $G$ has a natural Lie group structure.
We present an explicit formula for the group multiplication in the right trivialization
and for the group 2-cocycle describing the abelian Lie group extension $\g\to J^{k}G\to J^{k-1}G$.
\end{abstract}

\let\thefootnote\relax
\let\thefootnote\relax


\section{Introduction}

The jet bundles $J^kM\to M$ and higher order tangent bundles $T^kM\to M$
over a smooth manifold $M$ are examples of natural operations in the sense of \cite{KMS}.
When applied to a Lie group $G$, these constructions provide new Lie groups $J^kG$ and $T^kG$.
In this article we provide a formula for the group structure on $J^kG$ in terms of the Lie bracket
and we compare it with a similar formula for the group structure on $T^kG$ \cite{B}.
The tangent functor possesses a natural section for
$T^kG\to T^{k-1}G$, so $T^kG$ is a semidirect product of $T^{k-1}G$,
while the jet functor provides an abelian extension $J^kG\to J^{k-1}G$. 

The manifold of $k$-jets of smooth curves in a Lie group $G$ is a fiber bundle $J^kG\to G$,  
called the $k$-th order jet bundle  \cite{KMS}.
Denoting by $j^kc$ the $k$-jet at $0$ of the curve $c$,
the multiplication $(j^kc)(j^kb):=j^k(cb)$ in $J^kG$ doesn't depend on the representing curves
and defines a Lie group structure on $J^kG$.
Its Lie algebra $J^k\g$ is isomorphic to $\g\otimes_\RR\RR[X]/(X^{k+1})$,
the scalar extension of the Lie algebra $\g$ by the truncated polynomial ring $\RR[X]/(X^{k+1})$.
As a vector space it is isomorphic to $(k+1)$ copies of $\g$.

The main result of the paper is Theorem \ref{jet} which gives the  expression of the group multiplication
on the jet bundle $J^kG$ in the right trivialization \cite{GHMRV}
\begin{equation*}
j^kc\in J^kG\mapsto \left(c(0),(\de^rc)(0),(\de^rc)'(0),\dots,(\de^rc)^{(k-1)}(0)\right)\in G\x\g^k,
\end{equation*}
where $c$ is a smooth curve in $G$ with right logarithmic derivative $\de^rc=c'c^{-1}$.
One gets the multiplication law 
\begin{equation*}
\left(g,(x_n)_{1\le n\le k}\right)\left(h,(y_n)_{1\le n\le k}\right)=\left(gh,
\left(x_n+\sum_{\la\in\P_{n}}\ad_{x_{i_{\ell-1}}}\dots\ad_{x_{i_1}}\Ad_gy_{i_\ell}\right)_{1\le n\le k}\right)
\end{equation*}
and the inverse 
\begin{equation*}
\left(g,(x_n)_{1\le n\le k}\right)^{-1}=\left(g^{-1},\left( \sum_{\lambda\in\P_n} (-1)^\ell\Ad_{g^{-1}}\ad_{x_{i_1}}\dots\ad_{x_{i_{\ell-1}}}x_{i_\ell}\right)_{1\le n\le k}\right),
\end{equation*}
both involving sums over all anti-lexicographically ordered partitions 
$\la_1\cup\dots\cup\la_\ell=\{1,\dots,n\}$ 
(disjoint union) with $i_r$ the cardinality of the subset $\la_r$ for $r=1,\dots,\ell$.
The number of anti-lexicographically ordered partitions with fixed cardinalities $(i_1,\dots,i_\ell)$ is
\begin{equation}\label{comb}
N_{(i_1,\dots,i_\ell)}=\binom{i_1+\dots+i_\ell-1}{i_\ell-1}
\binom{i_1+\dots i_{\ell-1}-1}{i_{\ell-1}-1}\dots\binom{i_1+i_2-1}{i_2-1},
\end{equation}
hence the $n$-th component in the above formulas can be written respectively  as
\begin{gather*}
z_n=x_n+\sum_{i_1+\dots +i_\ell=n}N_{(i_1,\dots,i_\ell)}\ad_{x_{i_{\ell-1}}}\dots\ad_{x_{i_1}}\Ad_gy_{i_\ell}
\\
w_n=\sum_{i_1+\dots +i_\ell=n}(-1)^\ell N_{(i_1,\dots,i_\ell)} \Ad_{g^{-1}}\ad_{x_{i_1}}\dots\ad_{x_{i_{\ell-1}}}x_{i_\ell}.
\end{gather*}

The formulas involving partitions remind of the expressions for the multiplication and the inverse 
for the $k$-th order tangent group $T^kG$ in the right trivialization \cite{B}.
The reason is that the $k$-th order jet bundle $J^kG$ can be identified with the subgroup $(T^kG)^{S_k}$ of $T^kG$,
the fixed point set under the natural action of the permutation group $S_k$ 
(see Theorem \ref{first}).

All these formulas have no denominators, so one can work with geometric objects over general fields \cite{B}.
The identity fibers $J_k(\g)$ and $G_k(\g)$ of the fiber bundles $J^kG$ and $T^kG$ are polynomial groups depending only on the Lie bracket on $\g$. The dimension of $J_k(\g)$ is $k\dim\g$, while the dimension of $G_k(\g)$ is $(2^k-1)\dim\g$.
Both assignments give functorial maps 
from Lie algebras to polynomial groups, even in the more general case when $\g$  is a Leibniz algebra
\cite{D}.
Moreover, the group multiplication is in both cases affine in the second argument,
so we actually get affine near-ring structures, slight generalizations of near-ring structures
(where the group multiplication is linear in one of the arguments) \cite{P}. 

In the right trivialization, the first order jet bundle $J^1G$ coincides with the tangent bundle $TG$, so
it is the semidirect product $G\ltimes\g$
for the adjoint action. 
The second order jet bundle $J^2G$ is an abelian extension of $G\ltimes \g$ by $\g$
with characteristic group cocycle
\[
c((g,x),(h,y))=\ad_x\Ad_gy,
\]
and non-trivial Lie algebra cocycle 
\begin{equation*}
\si((\xi,x),(\et,y))=2[x,y].
\end{equation*}
This is a special case of Proposition \ref{csi}, where the characteristic cocycles 
for abelian extensions $\g\to J^kG\to J^{k-1}G$ associated
to higher order jet bundles are computed.

\paragraph{Acknowledgments} The author is grateful to Wolfgang Bertram and to the referee for extremely helpful suggestions.
This work was supported by a grant of the Romanian National Authority for Scientific Research, 
CNCS UEFISCDI, project number PN-II-ID-PCE-2011-3-0921.


\section{Group multiplication in jet bundles}

This section contains the main theorem of this paper, namely the expression of the multiplication on 
the $k$-th order jet bundle $J^kG$ 
in right trivialization. 

\begin{theorem}{\rm \cite{GHMRV}}
The right trivialization of the jet bundle $J^kG\to G$:
\begin{equation}\label{rtr}
j^kc\in J^kG\mapsto \left(c(0),(\de^rc)(0),(\de^rc)'(0),\dots,(\de^rc)^{(k-1)}(0)\right)\in G\x\g^k,
\end{equation}
where $c$ is a smooth curve in $G$ and $j^kc$ its $k$-jet at 0,
is an isomorphism of bundles,
whose inverse assigns to $(g,x_1,x_2,\dots,x_k)\in G\x\g^k$ the $k$-jet of the curve $c$ in $G$, 
uniquely defined by $c(0)=g$ and $\de^rc(t)=x'(t)$ for the Lie algebra curve 
$$
x(t)=tx_1+\frac{t^2}{2!}x_2+\dots+\frac{t^{k}}{k!}x_k.
$$
\end{theorem}

\begin{remark}
{\rm When $k>2$, the $k$-jet $j^kc$ with right trivialization $(g,x_1,\dots, x_k)$
doesn't coincide in general with the $k$-jet $j^kb$ of the curve $b(t)=(\exp x(t))g$, 
where $x$ is the Lie algebra curve defined above.
Indeed, the right logarithmic derivative of $b$ is 
$$
\de^rb(t)=\de^r\exp x(t)=\sum_{j=0}^\oo\frac{1}{(j+1)!}\ad^j_{x(t)}x'(t),
$$
so $\de^rb(0)=x_1$, $(\de^rb)'(0)=x_2$, but the higher order derivatives get additional terms $(\de^rb)''(0)=x_3+\frac12\ad_{x_1}x_2$,
 $(\de^rb)'''(0)=x_4+\ad_{x_1}x_3+\frac12\ad^2_{x_1}x_2$.}
\end{remark}

The multiplication $(j^kc)(j^kb):=j^k(cb)$ doesn't depend on the representing curves
and defines a Lie group structure on the $k$-th order jet bundle $J^kG\to G$.
Denoting  the identity fiber by $J_k(\g)$,
the Lie group $J^kG$ is a semidirect product of $G$ and $J_k(\g)$.
 
Let us introduce some notation needed for expressing the group multiplication on $J^kG$ in the right trivialization.

\begin{definition}
{\rm Let $\P_n$ denote the set of all partitions $\la=\la_1|\dots |\la_\ell$ of $\{1,\dots ,n\}$,
which means that $\{1,2,\dots,n\}=\la_1\cup\dots\cup\la_\ell$ disjoint union of sets.
We call $\ell(\la)=\ell$ the length of the partition.
We order each partition anti-lexicographically: the ordering is done from right to left,
always choosing the subset that contains the highest available number. 
For $n=3$ there are 5 such anti-lexicographically ordered partitions $\P_3=\{1|2|3,12|3,2|13,1|23,123\}$.

The cardinality of $\la_r$ is denoted by $i_r=|\la_r|$.
Of course $n=i_1+\dots+i_\ell$, so each element $\la\in\P_n$ determines an ordered decomposition of the number $n$.
There are elements in $\P_n$ that determine the same ordered decomposition of $n$, 
e.g. $2|13$ and $1|23$ both determine the decomposition $3=1+2$.
}
\end{definition}

\begin{lemma}\label{number}
The number of anti-lexicographically ordered partitions $\la=\la_1|\dots |\la_\ell$ of $\{1,\dots ,n\}$ 
with fixed cardinalities $i_1,\dots,i_\ell$ is $N_{(i_1,\dots,i_\ell)}$ from \eqref{comb}.
\end{lemma}

\begin{proof}
Because of the anti-lexicographic ordering, the subset $\la_\ell$ must contain the element $n$, 
while the other $i_\ell-1$ elements can be chosen in
$\binom{n-1}{i_\ell-1}$ ways. The subset $\la_{\ell-1}$ must contain the biggest of the remaining elements,
so we have $\binom{n-i_\ell-1}{i_{\ell-1}-1}$ choices. Similarly, we get 
$\binom{n-i_\ell-\dots-i_{r+1}-1}{i_r-1}$ choices for the elements of $\la_r$, $r=1,\dots,\ell$.
Their product gives the expression of $N_{(i_1,\dots,i_\ell)}$ from \eqref{comb} because  of the identity 
$i_1+\dots+i_\ell=n$.
\end{proof}

\begin{remark}\label{back}
{\rm 
From each anti-lexicographically ordered partition $\la=\la_1|\dots|\la_\ell\in\P_{n}$ 
one derives new anti-lexicographically ordered partitions 
$$
\la^{[m]}\in\P_{n+1},\quad m={0,\dots,\ell}
$$ 
by the following procedure: first we increase by 1 each element of $\la$, 
thus getting a partition  of the set $\{2,\dots,n+1\}$, which we denote $\la^+$,
then we adjoin the element 1 to the partition $\la^+$.
There are $m+1$ different ways to do this, namely
by placing the element $1$ in the $m$-th subset of $\la^{+}$ one obtains
the partition $\la^{[m]}$ of length $\ell$,
while by placing the element $1$ in a separate subset (in front),
one obtains the partition $\la^{[0]}$ of length $\ell+1$.
In all cases, the new partition is again anti-lexicographically ordered.
For instance, the three derived partitions of $\la=2|13$ are $\la^{[0]}=1|3|24$, $\la^{[1]}=3|124$, and $\la^{[2]}=13|24$.

Each anti-lexicographically ordered partition in $\P_{n+1}$ is a derived partition 
for a unique anti-lexicographically ordered partition in $\P_n$, 
namely the one obtained by removing the element 1
and lowering each remaining element by 1,
e.g. $2|13=\la^{[2]}\in\P_3$ for $\la=1|2\in\P_2$.}
\end{remark}

\begin{theorem}\label{jet}
The multiplication law in $J^kG$, identified with $G\x\g^k$ 
in the right trivialization \eqref{rtr}, is
$(g,x_1,\dots,x_k)(h,y_1,\dots,y_k)=(gh,z_1,\dots,z_k)$,
where each $z_n\in\g$ is given by
\begin{align}\label{zet}
z_n&=x_n+\sum_{\la\in\P_{n}}\ad_{x_{i_{\ell-1}}}\dots\ad_{x_{i_1}}\Ad_gy_{i_\ell}
\nonumber\\&=x_n+\sum_{\ell=1}^n
\sum_{i_1+\dots +i_\ell=n}N_{(i_1,\dots,i_\ell)}\ad_{x_{i_{\ell-1}}}\dots\ad_{x_{i_1}}\Ad_gy_{i_\ell},
\end{align}
and the inverse in $J^kG$ is
$(g,x_1,\dots,x_k)^{-1}=(g^{-1},w_1,\dots,w_k)$, where each $w_n\in\g$ is given by
\begin{align}\label{wet}
w_n &= \sum_{\lambda\in\P_n} (-1)^\ell\Ad_{g^{-1}}\ad_{x_{i_1}}\dots\ad_{x_{i_{\ell-1}}}x_{i_\ell}\nonumber
\\&=\sum_{\ell=1}^n\sum_{i_1+\dots +i_\ell=n}(-1)^\ell N_{(i_1,\dots,i_\ell)} \Ad_{g^{-1}}\ad_{x_{i_1}}\dots\ad_{x_{i_{\ell-1}}}x_{i_\ell}.
\end{align}
$\P_n$ denotes the set of anti-lexicographically ordered partitions
$\la=\la_1|\dots|\la_\ell$ of $\{1,\dots,n\}$, $i_r$ denotes the cardinality of $\la_r$, and
$N_{(i_1,\dots,i_\ell)}$ is given by \eqref{comb}.
\end{theorem}

\begin{proof}
In order to prove 
the multiplication formula \eqref{zet} we consider the curves $x$ and $y$ in the Lie algebra $\g$,  
\begin{align*}
x(t)=tx_1+\frac{t^2}{2!}x_2+\dots+\frac{t^{k}}{k!}x_k,\text{ resp. }y(t)=ty_1+\frac{t^2}{2!}y_2+\dots+\frac{t^{k}}{k!}y_k,
\end{align*}
and the curves $c$ and $b$ in the Lie group $G$, uniquely defined by $\de^rc=x'$ and $c(0)=g$, 
\resp $\de^rb=y'$ and $b(0)=h$.
 In the right trivialization \eqref{rtr} we identify
$j^kc=(g,x'(0),\dots,x^{(k)}(0))=(g,x_1,\dots,x_k)$ \resp $j^kb=(h,y_1,\dots,y_k)$, 
so the multiplication is
\[
(g,x_1,\dots,x_k)(h,y_1,\dots,y_k)=j^k(cb)=(gh,z'(0),z''(0),\dots,z^{(k)}(0))
\]
for $z'=\de^r(cb)=\de^rc+\Ad_c\de^rb=x'+\Ad_cy'$. 

It remains to prove that $z_n=z^{(n)}(0)$.
To each anti-lexicographically ordered partition $\la\in\P_n$ with $|\la_r|=i_r$ we assign a curve $F_\la$ in $\g$ defined by
\[
F_\la(t):=\ad_{x^{(i_{\ell-1})}(t)}\dots\ad_{x^{(i_{1})}(t)}\Ad_{c(t)}y^{(i_\ell)}(t).
\]
Its derivative can be written as a sum of  terms of the same type:
\begin{align}\label{fla}
F_\la'&=\left(\ad_{x^{(i_{\ell-1})}}\dots\ad_{x^{(i_{1})}}\Ad_{c}y^{(i_\ell)}\right)'
=\ad_{x^{(i_{\ell-1})}}\dots\ad_{x^{(i_{1})}}\ad_{x'}\Ad_{c}y^{(i_\ell)}\nonumber\\
&+\sum_{m=1}^{\ell-1}\ad_{x^{(i_{\ell-1})}}\dots\ad_{x^{(i_m+1)}}\dots\ad_{x^{(i_{1})}}\Ad_{c}y^{(i_\ell)}
+\ad_{x^{(i_{\ell-1})}}\dots\ad_{x^{(i_{1})}}\Ad_{c}y^{(i_\ell+1)},
\end{align}
using at step two the fact that 
$$(\Ad_cy^{(i)})'=\ad_{\de^rc}\Ad_cy^{(i)}+\Ad_cy^{(i+1)}=\ad_{x'}\Ad_cy^{(i)}+\Ad_cy^{(i+1)}.$$

Since the ordered decomposition of the number $n$ 
induced by the anti-lexicographic ordered partition $\la\in\P_n$ is $n=i_1+\dots+i_\ell$,
the ordered decompositions of $n+1$ induced by the derived partitions 
$\la^{[0]},\dots,\la^{[m]},\dots,\la^{[\ell]}\in\P_{n+1}$
are $n+1=1+i_1+\dots+i_\ell$, $\dots,n+1=i_1+\dots+(i_m+1)+\dots+i_\ell$, ..., $n+1=i_1+\dots+(i_\ell+1)$.
This ensures that each term in the expression \eqref{fla} of $F'_\la$ corresponds to one of the derived partitions of $\la$, so
\begin{align*}
F_\la'(t)&=\sum_{m=0}^{\ell}F_{\la^{[m]}}(t).
\end{align*}

We are now ready to compute the higher order derivatives of $z'=x'+\Ad_cy'$. We will prove by induction that:
\begin{equation}\label{zen}
z^{(n)}(t)=x^{(n)}(t)+\sum_{\la\in\P_{n}}F_\la(t).
\end{equation}
For $n=2$ the identity holds because $z''=(x'+\Ad_cy')'=x''+\ad_{x'}\Ad_cy'+\Ad_cy''$ and $\P_2$ consists of only two partitions
$1|2$ and $12$.
Assuming that \eqref{zen} holds for $n-1$, we compute $z^{(n)}$:
\begin{align*}
z^{(n)}&
=x^{(n)}+\sum_{\la\in\P_{n-1}}F_\la'
=x^{(n)}+\sum_{\la\in\P_{n-1}}\sum_{m=0}^{|\la|}F_{\la^{[m]}}=x^{(n)}+\sum_{\rho\in\P_{n}}F_\rho.
\end{align*}
In the last step we use Remark \ref{back}: each element $\rho\in\P_{n}$ is obtained from a unique 
element $\la\in\P_{n-1}$ through derivation.
This ends the proof of  \eqref{zen} by induction, and the first part of \eqref{zet} 
follows by evaluation at 0. For the second part of \eqref{zet} we apply Lemma \ref{number}.

To show the inversion formula \eqref{wet}, we consider again a Lie algebra curve $x(t)=tx_1+\frac{t^2}{2!}x_2+\dots+\frac{t^{k}}{k!}x_k$,
and the Lie group curve $c$ uniquely defined by $\de^rc=x'$ and $c(0)=g$.
In the right trivialization \eqref{rtr} we have
$j^kc=(g,x_1,\dots,x_k)$, so
\[
(g,x_1,\dots,x_k)^{-1}=j^k(c^{-1})=(g^{-1},w'(0),w''(0),\dots,w^{(k)}(0))
\]
for $w'=\de^r(c^{-1})=-\Ad_{c^{-1}}\de^rc=-\Ad_{c^{-1}}x'$. 
We show by induction that 
\begin{equation}\label{wen}
w^{(n)}=\sum_{\la\in\P_n}(-1)^\ell\Ad_{c^{-1}}\ad_{x^{(i_1)}}\dots\ad_{x^{(i_{\ell-1})}}x^{(i_\ell)}.
\end{equation}
We denote by $E_\la$ the Lie algebra curve $(-1)^\ell\Ad_{c^{-1}}\ad_{x^{(i_1)}}\dots\ad_{x^{(i_{\ell-1})}}x^{(i_\ell)}$.
For $n=2$ we compute $w''=-(\Ad_{c^{-1}}x')'=-\Ad_{c^{-1}}x''+\Ad_{c^{-1}}\ad_{x'}x'=E_{12}+E_{1|2}$.
Assuming that \eqref{wen} holds for $n-1$, we compute:
\begin{align*}
w^{(n)}&
=\sum_{\la\in\P_{n-1}}E_\la'
=\sum_{\la\in\P_{n-1}}\sum_{m=0}^{|\la|}E_{\la^{[m]}}=\sum_{\rho\in\P_{n}}E_\rho.
\end{align*}
where we use at step two the identity
\begin{align*}
(E_\la)'&=\left((-1)^\ell\Ad_{c^{-1}}\ad_{x^{(i_1)}}\dots\ad_{x^{(i_{\ell-1})}}x^{(i_\ell)}\right)'
=(-1)^{\ell+1}\Ad_{c^{-1}}\ad_{x'}\ad_{x^{(i_1)}}\dots\ad_{x^{(i_{\ell-1})}}x^{(i_\ell)}\\
&+\sum_{m=1}^{\ell}(-1)^{\ell}\Ad_{c^{-1}}\ad_{x^{(i_1)}}\dots\ad_{x^{(i_m+1)}}\dots\ad_{x^{(i_{\ell-1})}}x^{(i_\ell)}
=\sum_{m=0}^{\ell}E_{\la^{[m]}}.
\end{align*}
This ends the proof of \eqref{wen} by induction.

Evaluation at 0 gives us the first part of \eqref{wet}, while the second part of \eqref{wet} 
follows immediately with Lemma \ref{number}.
\end{proof}

\begin{example}
For convenience of the reader we expand the multiplication in the identity fiber $J_4(\g)$ of $J^4G$:
\begin{align*}
(x_1,x_2,x_3,x_4)(y_1,y_2,y_3,y_4)
=(&x_1+y_1,x_2+y_2+\ad_{x_1}y_1,\\
&x_3+y_3+2\ad_{x_1}y_2+\ad_{x_2}y_1+\ad_{x_1}^2y_1,\\
&x_4+y_4+3\ad_{x_1}y_3+3\ad_{x_2}y_2+3\ad_{x_1}^2y_2\\
&+\ad_{x_3}y_1+2\ad_{x_2}\ad_{x_1}y_1+\ad_{x_1}\ad_{x_2}y_1+\ad_{x_1}^3y_1).
\end{align*}
\end{example}

\begin{remark}\label{leftrem}
{\rm
The counterpart of \eqref{zet} and \eqref{wet} in the left trivialization of $J^kG$ 
(the bijection \eqref{rtr} with the right logarithmic derivative $\de^rc$ 
replaced by the left logarithmic derivative $\de^lc=c^{-1}c'$) 
are
\begin{equation*}
z_n
=y_n+\sum_{i_1+\dots+i_\ell={n}}(-1)^{\ell-1}N_{(i_1,\dots,i_\ell)}\ad_{y_{i_{\ell-1}}}\dots\ad_{y_{i_1}}\Ad_{h^{-1}}x_{i_\ell},
\end{equation*}
and 
\begin{equation*}
w_n
=-\sum_{i_1+\dots+i_\ell={n}}N_{(i_1,\dots,i_\ell)}\Ad_{g}\ad_{x_{i_1}}\dots\ad_{x_{i_{\ell-1}}}x_{i_\ell}.
\end{equation*}
}
\end{remark}

A {\it pure element} of $J^kG$ is a $k$-jet whose expression $(g,x_i)$ in the right trivialization
contains a single non-zero Lie algebra element.
The inverse of the pure element $(e,x_i)$ is $(e,-x_i)$.
The following identity concerning multiplication of pure elements was already obtained in \cite{B}
\[
(e,x_1)(e,y_1)=(e,x_1+y_1,\ad_{x_1}y_1,\ad_{x_1}^2y_1,\dots,\ad_{x_1}^{k-1}y_1).
\]
It is not hard to see that the multiplication with a pure element $(e,y_i)$, where $i>1$, gives
\begin{align*}
(e,x_1)(e,y_i)=\left(e,x_1,y_i,\binom{i}{1}\ad_{x_1}y_i,\binom{i+1}{2}\ad_{x_1}^2y_i,\dots,\binom{k-1}{k-i}\ad_{x_1}^{k-i}y_i\right).
\end{align*}
The multiplication of two arbitrary pure elements is the content of the following corollary.

\begin{corollary}\label{pure}
In the right trivialized jet bundle $J^kG$, the product of two pure elements $(e,x_i)$ and $(e,y_j)$ with $1\le i<j\le k$ is
\[
(e,x_i)(e,y_j)=(e,x_i,y_j,\frac{(i+j-1)!}{i!(j-1)!}\ad_{x_i}y_j,\dots,
\frac{(ni+j-1)!}{n!(i!)^n(j-1)!}\ad_{x_i}^{n}y_j),
\]
where $n$ is the biggest natural number that satisfies $ni+j\le k$. 
The components known to be zero were not written.
Similar formulas hold for $i= j$ and $i>j$.
\end{corollary}

\begin{proof}
It is enough to compute
\begin{multline*}
N_{(i,\dots,i,j)}=\binom{ni+j-1}{j-1}
\binom{ni-1}{i-1}\dots\binom{2i-1}{i-1}\\
=\frac{(i+1)\dots(2i-1)(2i+1)\dots(ni-1)(ni+1)\dots(ni+j-1)}
{((i-1)!)^{n-1}(j-1)!}
=\frac{(ni+j-1)!}{n!(i!)^n(j-1)!}
\end{multline*}
and to apply Theorem \ref{jet}.
\end{proof}


\section{Group multiplication in $T^kG$}\label{three}

In this section we present the multiplication on the trivialized $k$-th order tangent group $T^kG$ \cite{B}
and we show that the $k$-th order jet group $J^kG$ is isomorphic to the subgroup of fixed points
of $T^kG$ under the obvious action of the permutation group $S_k$.

The structure of higher order tangent groups was investigated in \cite{B} Section 24.
In the right trivialization we have a tower of semidirect products:
\[
T^{k-1}\g\to T^kG=T(T^{k-1}G)\to T^{k-1}G,
\]
where $T^{k-1}\g$ denotes the Lie algebra of $T^{k-1}G$.
Using infinitesimal units $\ep_1,\dots,\ep_k$ to keep track of each extension in the tower,
\ie $T^kG=T^{k-1}G\x \ep_kT^{k-1}\g$,
we get the right trivialization of the $k$-th order tangent bundle $T^kG$ 
\begin{align*}
TG&= G\x\ep_1\g\\
T^2G&= TG\x\ep_2 T\g=G\x\ep_1\g\x\ep_2\g\x\ep_2\ep_1\g
\\
\dots
\\
T^kG&= G\x\oplus_{\al\in I_k^*}\ep^\al\g
\end{align*}
where $I_k$ denotes the power set of $\{1,\dots,k\}$ and $I_k^*=I_k-\{\emptyset\}$.
For any multi-index $\al=\{\al_1,\dots,{\al_n}\}\in I_k^*$ we denote $\ep^\al:=\ep_{\al_n}\dots\ep_{\al_1}$.
In particular $T^k\g=\oplus_{\al\in I_k}\ep^\al\g$ as a vector space.

An alternative notation is used in \cite{B}: the multi-index $\al\in I_k$
is identified with its characteristic function in $\{0,1\}^k$, written as a string of 0's and 1's,
for instance the string corresponding to the multi-index $\{2,3\}\in I_3$ is $(110)$.
We state below the Theorem 24.7 from \cite{B} in our notation.

\begin{theorem}{\rm\cite{B}}\label{gro}
The group multiplication for the $k$-th order tangent group $T^kG$ in the right trivialization is
\[
(g,(\ep^\al x_\al)_{\al\in I_k^*})(h,(\ep^\al y_\al)_{\al\in I_k^*})=(gh,(\ep^\al z_\al)_{\al\in I_k^*})
\]
with
\begin{equation}\label{zal}
z_\al=x_\al+\sum_{\la\in\P(\al)}\ad_{x_{\la_{\ell-1}}}\dots\ad_{x_{\la_1}}\Ad_gy_{\la_\ell}
\end{equation}
where $\P(\al)$ denotes the set of all anti-lexicographically ordered partitions 
of the subset $\al\subset \{1,\dots,k\}$
and $\ell$ is the length of the partition $\la$.
The inverse is given by the formula 
$$
(g,(\ep^\al x_\al)_{\al\in I_k^*})^{-1}=(g^{-1},(\ep^\al w_\al)_{\al\in I_k^*})
$$
with
\begin{equation}\label{wal}
w_\al=\sum_{\la\in\P(\al)}(-1)^\ell\Ad_{g^{-1}}\ad_{x_{\la_{\ell-1}}}\dots\ad_{x_{\la_1}}x_{\la_\ell}.
\end{equation}
\end{theorem}

\begin{remark}
{\rm
The analogues of \eqref{zal} and \eqref{wal} in the left trivialization of $T^kG$ are
\begin{equation}\label{zal2}
z_\al=y_\al+\sum_{\la\in\P(\al)}(-1)^{\ell-1}\ad_{y_{\la_{\ell-1}}}\dots\ad_{y_{\la_1}}\Ad_{h^{-1}}x_{\la_\ell}
\end{equation}
and 
\begin{equation}\label{wal2}
w_\al =-\sum_{\la\in\P(\al)}\Ad_{g}\ad_{x_{\la_{1}}}\dots\ad_{x_{\la_{\ell-1}}}x_{\la_\ell},
\end{equation}
}
\end{remark}

We consider the obvious left action of the symmetric group $S_k$ on $T^kG$ 
by permuting the infinitesimal units $\ep_1,\dots,\ep_k$ \cite{D}:
\begin{equation}\label{act}
\si\cdot (g,(\ep^\al x_\al)_{\al\in I_k^*})=(g,(\ep^{\si(\al)} x_{\al})_{\al\in I_k^*}),\quad\si\in S_k.
\end{equation}
In contrast to the action (24.4) from \cite{B},
this action is not compatible with the group multiplication \eqref{zal}
(a permutation $\si\in S_k$ might perturb the anti-lexicographic ordering of $\la\in\P_n$),
but we still have the following result:

\begin{proposition}\label{first}
With respect to the action of $S_k$ given by \eqref{act}, the fixed point set $(T^kG)^{S_k}$ is a subgroup of the tangent group $T^kG$,
isomorphic to the jet group $J^kG$.
\end{proposition}

\begin{proof}
Elements $(g,(\ep^\al x_\al)_{\al\in I_k^*})$ of $(T^kG)^{S_k}$ are characterized by $x_\al=x_\be$
for all $\al,\be\in I_k^*$ with the same cardinality,
hence there exist $x_1,\dots,x_n\in\g$ such that $x_\al=x_n$ for ${|\al|}=n$.
Let $(g,(\ep^\al x_\al)_{\al\in I_k^*}),(h,(\ep^\al y_\al)_{\al\in I_k^*})\in  (T^kG)^{S_k}$
with $x_\al=x_{|\al|}$ and $y_\al=y_{|\al|}$.
From the group multiplication \eqref{zal} on $T^kG$ we get
\begin{align*}
z_\al
=x_{|\al|}+\sum_{\la\in\P(\al)}\ad_{x_{|\la_{\ell-1}|}}\dots\ad_{x_{|\la_1|}}\Ad_gy_{|\la_\ell|}.
\end{align*}
For each $\al\in I_k^*$ with $|\al|=n$,
there exists a unique strictly increasing bijection $\ph:\{1,\dots,n\}\to\al=\{\al_1,\dots,\al_n\}$.
It induces canonically a 1-1 correspondence between the sets  $\P_n$ and $\P(\al)$ of
anti-lexicographically ordered partitions, bijection that  preserves the length of the partition as well as 
the cardinality of each subset. 
We get that
\begin{align}\label{long}
z_\al=x_n+\sum_{\la\in\P_n}\ad_{x_{|\la_{\ell-1}|}}\dots\ad_{x_{|\la_1|}}\Ad_gy_{|\la_\ell|},
\end{align}
hence each $z_\al$ depends only on $|\al|=n$, thus showing that 
$$
(g,(\ep^\al x_\al)_{\al\in I_k^*})(h,(\ep^\al y_\al)_{\al\in I_k^*})=(gh,(\ep^\al z_\al)_{\al\in I_k^*})\in (T^kG)^{S_k}.
$$ 
In a similar manner, using \eqref{wal}, one sees that the inverse $(g,(\ep^\al x_\al)_{\al\in I_k^*})^{-1}\in (T^kG)^{S_k}$, so $(T^kG)^{S_k}$ is a subgroup of $T^kG$.

The identification of an element $(g,(\ep^\al x_\al)_{\al\in I_k^*})$ in $(T^kG)^{S_k}$  
with the corresponding element $(g,(x_n)_{1\le n\le k})$ in $J^kG$,
where $x_n=x_\al$ for any $\al\in I_k^*$ with $|\al|=n$,
provides a bijection between these groups.
From the two expressions \eqref{zet} and \eqref{long} of the group multiplication in the right trivialization,
it is easy to see that this bijection is a group isomorphism.
\end{proof}

\begin{example}
{\rm
The multiplication on $T^3G$ is
\begin{gather*}
(e,(\ep^\al x_\al)_{\al\in I_3^*})(e,(\ep^\al y_\al)_{\al\in I_3^*})\\
=(e,\ep^1(x_1+y_1),\ep^2(x_2+y_2),
\ep^{12}(x_{12}+y_{12}+\ad_{x_1}y_2),\\\ep^3(x_3+y_3),
\ep^{13}(x_{13}+y_{13}+\ad_{x_1}y_3),\ep^{23}(x_{23}+y_{23}+\ad_{x_2}y_3),\\
\ep^{123}(x_{123}+y_{123}+\ad_{x_1}y_{23}+\ad_{x_2}y_{13}+\ad_{x_{12}}y_3
+\ad_{x_2}\ad_{x_1}y_3))
\end{gather*}
while the multiplication on $J^3G$ can be written as
\begin{gather*}
(e,x_1,x_2,x_3)(e,y_1,y_2,y_3)\\
=(e,x_1+y_1,x_2+y_2+\ad_{x_1}y_1,
x_3+y_3+2\ad_{x_1}y_2+\ad_{x_2}y_1+\ad_{x_1}^2y_1).
\end{gather*}
We observe that all elements of $T^3G$ can be decomposed in products of pure elements
\begin{gather*}
(e,(\ep^\al x_\al)_{\al\in I_3^*})
=(e,\ep^{123}x_{123})(e,\ep^{23}x_{23})(e,\ep^{13}x_{13}) (e,\ep^3x_3)
 (e,\ep^{12}x_{12})(e,\ep^2 x_2) (e,\ep^1x_1),   
\end{gather*}
a property that holds for all $T^kG$ \cite{B}.
There is no analogous decomposition for $J^3G$:
\[
(e,x_3)(e,x_2)(e,x_1)=(e,x_1,x_2,x_3+\ad_{x_2}x_1]).
\]
}
\end{example}


\paragraph{Jet functor.}
The identity fiber of $T^kG\to G$ is a subgroup of $T^kG$, denoted by $G_k(\g)$ 
and identified with $\oplus_{\al\in I_k^*}\ep^\al\g$,
so its dimension is $(2^k-1)\dim\g$.
The identity fiber $J_k(\g)$ of $J^kG\to G$ is a subgroup  of $J^kG$, 
identified with the direct sum of $k$ copies of $\g$,
so its dimension is $k\dim\g$.

As before we consider the natural action of the symmetric group $S_k$ on $G_k(\g)$ by permuting the infinitesimal units $\ep_1,\dots,\ep_k$:
\[
\si\cdot (\ep^\al x_\al)_{\al\in I_k^*}=(\ep^{\si(\al)} x_{\al})_{\al\in I_k^*},\quad\si\in S_k.
\]
The fixed point set $G_k(\g)^{S_k}$, identified with $J_k(\g)$,
is a subgroup of $G_k(\g)$, 
Elements $(\ep^\al x_\al)_{\al\in I_k^*}$ of the fixed point set are characterized by $x_\al=x_\be$
for all $\al,\be\in I_k^*$ with the same cardinality.

Theorem \ref{gro} resp. Theorem \ref{jet} provide the expressions for the group multiplication 
and the inverse of an element on $G_k(\g)$: 
\begin{gather}\label{mul}
(\ep^\al x_\al)_{\al\in I_k^*}(\ep^\al y_\al)_{\al\in I_k^*}=\left(\ep^\al \left(x_\al+\sum_{\la\in\P(\al)}\ad_{x_{\la_{\ell-1}}}\dots\ad_{x_{\la_1}}y_{\la_\ell}\right)\right)_{\al\in I_k^*}\\
(\ep^\al x_\al)_{\al\in I_k^*}^{-1}=\left(\ep^\al\left(\sum_{\la\in\P(\al)}(-1)^\ell\ad_{x_{\la_{\ell-1}}}\dots\ad_{x_{\la_1}}x_{\la_\ell}\right)\right)_{\al\in I_k^*}\nonumber,
\end{gather}
resp. on $J_k(\g)$: $\left(x_n\right)_{1\le n\le k}\left(y_n\right)_{1\le n\le k}
=\left(z_n\right)_{1\le n\le k}$ and $\left(x_n\right)_{1\le n\le k}^{-1}=\left(w_n\right)_{1\le n\le k}$,
\begin{gather}\label{jul}
z_n=x_n+\sum_{\la\in\P_{n}}\ad_{x_{i_{\ell-1}}}\dots\ad_{x_{i_1}}y_{i_\ell}
=x_n+\sum_{i_1+\dots +i_\ell=n}N_{(i_1,\dots,i_\ell)}\ad_{x_{i_{\ell-1}}}\dots\ad_{x_{i_1}}y_{i_\ell}
\\
w_n=\sum_{\lambda\in\P_n} (-1)^\ell\ad_{x_{i_1}}\dots\ad_{x_{i_{\ell-1}}}x_{i_\ell}=\sum_{i_1+\dots +i_\ell=n}(-1)^\ell N_{(i_1,\dots,i_\ell)} \ad_{x_{i_1}}\dots\ad_{x_{i_{\ell-1}}}x_{i_\ell}
\nonumber,
\end{gather}
where $\la$ is an anti-lexicographically ordered partition $\la_1|\dots|\la_\ell$
of $\al\in I_k^*$, resp. of $\{1,\dots,n\}$, of length $\ell$ and $i_r=|\la_r|$ for $r=1,\dots,\ell$.

The maps sending the Lie algebra $\g$ to the groups $G_k(\g)$ and $J_k(\g)$ are functorial.
Both are polynomial groups \cite{D}.
Moreover, the group multiplication is affine in the second argument,
so we actually get affine near-ring structures, slight generalizations of near-ring structures \cite{P}. 

\paragraph{Leibniz algebras.}
A left Leibniz algebra \cite{L} is a vector space $\g$ endowed with a bilinear map
$[\ ,\ ] : \g \times\g \to\g$  such that the left Leibniz identity holds
\[
[x,[y, z]] = [[x, y], z] + [y, [x, z]],\quad x, y, z\in\g.
\] 
This condition is equivalent to the fact that the left ad-actions are derivations of the bracket.
A Lie algebra is the same as a Leibniz algebra with antisymmetric bracket.

\begin{theorem}\label{didi}{\rm\cite{D}}
For any left Leibniz algebra $\g$, 
the multiplication \eqref{mul} endows $G_k(\g)$ with a group structure. 
Its fixed point set $J_k(\g)$ under the permutation action is a subgroup.
\end{theorem}

The multiplication formula \eqref{jul} on $J_k(\g)$ holds also for a Leibniz algebra $\g$.


\section{Abelian extension}\label{five}

This section is concerned with the group extension
$\g\to J^kG\to J^{k-1}G$ 
assigned to the right trivialization.
The following result is a direct consequence of Theorem \ref{jet}.

\begin{proposition}\label{csi}
The right trivialized $k$-th order jet bundle $J^{k}G=G\x\g^k$ is an abelian 
Lie group extension of the right trivialized $(k-1)$-th order jet bundle $J^{k-1}G=G\x\g^{k-1}$ by $\g$,
characterized by the group cocycle
\begin{align*}
c_k\left((g,x_1,\dots,x_{k-1}),(h,y_1,\dots,y_{k-1})\right)
&=\sum_{\la\in\P_{k}-\{{12\dots k}\}}\ad_{x_{i_{\ell-1}}}\dots\ad_{x_{i_1}}\Ad_gy_{i_\ell}\\
&=\sum_{\ell=2}^k\sum_{i_1+\dots i_\ell=k}N_{(i_1,\dots,i_\ell)}\ad_{x_{i_{\ell-1}}}\dots\ad_{x_{i_1}}\Ad_gy_{i_\ell}
\end{align*}
with $\la=\la_1|\dots|\la_\ell$ anti-lexicographically ordered and $i_r=|\la_r|$.

Its Lie algebra $J^{k}\g$ is an abelian extension of $J^{k-1}\g$ by $\g$ with Lie algebra cocycle
\[
\si_k\left((\xi,x_1,\dots,x_{k-1}),(\et,y_1,\dots,y_{k-1})\right)
=\sum_{i=1}^{k-1} \binom{k}{i}\left[x_{i},y_{k-i}\right].
\] 
\end{proposition}

\begin{proof}
The expression of the group cocycle $c_k$ follows immediately from the multiplication rule \eqref{zet}
rewritten as
\[
z_n=x_n+\Ad_g y_n+\sum_{\la\in\P_{k}-\{{12\dots k}\}}\ad_{x_{i_{\ell-1}}}\dots\ad_{x_{i_1}}\Ad_gy_{i_\ell}.
\]
The Lie algebra cocycle $\si_k$ can be obtained from the group cocycle $c_k$ by derivation:
\begin{equation}\label{sika}
\si_k=\partial_1\partial_2c_k(e,e)-\partial_2\partial_1c_k(e,e).
\end{equation}
We will see that only the terms of the Lie group cocycle that
correspond to partitions $\la\in\P_{k}$ of length two provide non-zero terms 
for the Lie algebra cocycle. 

There are $\binom{k-1}{i}$ anti-lexicographically ordered partitions $\la\in\P_{k}$ of length $2$ 
such that $|\la_1|=i$ and $|\la_2|=k-i$, because the only requirement is that $k\in\la_2$.
Thus, splitting $\P_{k}-\{{12\dots k}\}$ into ordered partitions of length 2 and ordered partitions of length bigger than 2, 
the group cocycle can be written as
\begin{align*}
c_k\left((g,x_1,\dots,x_{k-1}),(h,y_1,\dots,y_{k-1})\right)
=\sum_{i=1}^{k-1}\binom{k-1}{i}\ad_{x_i}\Ad_gy_{k-i}+\sum_{p+q< k}\ad_{x_p}\ad_{x_q}R_{pq}^k,
\end{align*}
where $R_{pq}^k$ is of the form $\ad_{x_{j_m}}\dots\ad_{x_{j_1}}\Ad_gy_{j_0}$ with $m\ge 0$. Then
\begin{equation*}
\partial_1\partial_2c_k(e,e)\left((\xi,x_1,\dots,x_{k-1}),(\et,y_1,\dots,y_{k-1})\right)
=\sum_{i=1}^{k-1}\binom{k-1}{i}\ad_{x_{i}}y_{k-i}
\end{equation*}
and by \eqref{sika} the Lie algebra cocycle is
\begin{multline*}
\si_k\left((\xi,x_1,\dots,x_{k-1}),(\et,y_1,\dots,y_{k-1})\right)
=\sum_{i=1}^{k-1}\binom{k-1}{i}\ad_{x_{i}}y_{k-i}-\sum_{i=1}^{k-1}\binom{k-1}{i}\ad_{y_{i}}x_{k-i}\\
=\sum_{i=1}^{k-1}\left(\binom{k-1}{i}+\binom{k-1}{k-i}\right)\ad_{x_i}y_{k-i}=\sum_{i=1}^{k-1} \binom{k}{i}\left[x_{i},y_{k-i}\right],
\end{multline*}
as requested.
\end{proof}

\begin{remark}
{\rm 
Given two $\g$-modules $V$ and $W$,
each $\g$-invariant element $\ga$ in $\La^2 V^*\otimes W$ determines canonically a Lie algebra 2-cocycle
on the semidirect product $\g\ltimes V$
\[
\si_\ga((\xi,u),(\et,v))=\ga(u,v)\in W.
\]
Its cohomology class in $H^2(\g\ltimes V,W)$ is non-zero.
A special case is the $\g$-invariant element $[\ ,\ ]$ in $\La^2\g^*\otimes\g$,
which determines the Lie algebra 2-cocycle $\si_2$ that characterizes $J^2\g$.
}
\end{remark}

In the right trivialization, the $\ep^{12\dots k}\g$ component of $T^kG$ with addition is a normal abelian subgroup of $T^kG$,
hence we get an abelian Lie group extension
\[
\g\to T^kG\to T^kG/\g.
\]
The  jet group $J^{k-1}G=(T^{k-1}G)^{S_{k-1}}$ is isomorphic to the subgroup $(T^kG/\g)^{S_k}$
of fixed points, since the only multi-index with cardinality $k$, namely $12\dots k$, was divided out.
The pull-back of the abelian extension above by the inclusion $J^{k-1}G\subset T^kG/\g$
is nothing else but the abelian extension
\[
\g\to J^kG\to J^{k-1}G.
\]


CORNELIA VIZMAN, {Department of Mathematics,
West University of Timi\c soara,  Romania, Bd. V. P\^arvan 4, 300223-Timi\c soara, 
\texttt{vizman@math.uvt.ro}

\end{document}

\begin{gather*}
(g,x_1,x_2,x_3)(h,y_1,y_2,y_3)\\
=(gh,x_1+\Ad_gy_1,x_2+\Ad_gy_2+\ad_{x_1}\Ad_gy_1,\\
x_3+\Ad_gy_3+2\ad_{x_1}\Ad_gy_2+\ad_{x_2}\Ad_gy_1+\ad_{x_1}^2\Ad_gy_1).
\end{gather*}
We observe that all elements of $T^3G$ can be decomposed in products of pure elements
\begin{gather*}
(e,(\ep^\al x_\al)_{\al\in I_3^*})
=(e,\ep^{123}x_{123})(e,\ep^{23}x_{23})(e,\ep^{13}x_{13}) (e,\ep^3x_3)
 (e,\ep^{12}x_{12})(e,\ep^2 x_2) (e,\ep^1x_1),   
\end{gather*}

\begin{align*}
(g,x_1)(h,y_i)=(gh,x_1,\Ad_gy_i,i\ad_{x_1}\Ad_gy_i,\frac{i(i+1)}{2}\ad_{x_1}^2\Ad_gy_i,\dots,\binom{n+i-1}{n}\ad_{x_1}^{n}\Ad_gy_i).
\end{align*}


\begin{proof}
The proof is very similar to the proof of Theorem \ref{first}, since elements $(\ep^\al x_\al)_{\al\in I_k^*}$ of $G_k(\g)^{S_k}=J_k(\g)$ are characterized by  $x_\al=x_{|\al|}$.
Let $(\ep^\al x_\al)_{\al\in I_k^*},(\ep^\al y_\al)_{\al\in I_k^*}\in J_k(\g)$.
From the group multiplication \eqref{mul} on $G_k(\g)$ we get $(\ep^\al x_\al)_{\al\in I_k^*}(\ep^\al y_\al)_{\al\in I_k^*}=(\ep^\al z_\al)_{\al\in I_k^*}$ with
\[
z_\al
=x_{|\al|}+\sum_{\la\in\P_{|\al|}}\ad_{x_{|\la_{\ell-1}|}}\dots\ad_{x_{|\la_1|}}y_{|\la_\ell|}.
\] 
We use again the bijection between  $\P(\al)$ and  $\P_{|\al|}$
that preserves the length of the partition as well as the cardinality of each subset. 
Similarly, using \eqref{inv}, one gets $(\ep^\al x_\al)_{\al\in I_k^*}^{-1}=(\ep^\al w_\al)_{\al\in I_k^*}$ with
\[
w_\al=\sum_{\la\in\P_{|\al|}}(-1)^\ell\ad_{x_{\la_{\ell-1}}}\dots\ad_{x_{\la_1}}x_{\la_\ell}.
\] 
Hence all $z_\al$ and $w_\al$ depend only on $|\al|$, thus  $J_k(\g)\subset G_k(\g)$ is closed under
multiplication and under taking the inverse.
\end{proof}